\documentclass[twoside]{amsart}
\usepackage{amsmath,amsfonts,amsthm,mathrsfs,amscd}
\usepackage{hyperref}
\usepackage[numeric,initials]{amsrefs}
\usepackage{enumitem}
\usepackage{dcpic,pictexwd}
\usepackage{amsmath,amsfonts,amsthm,mathrsfs,amscd,stmaryrd}
\usepackage[numeric,initials]{amsrefs}

\newcommand{\abs}[1]{\lvert #1 \rvert}

\newcommand{\eps}{\varepsilon}

\newcommand{\fust}[1][f]{{#1}^{\ast}}

\newcommand{\finv}[1][f]{{#1}^{-1}}

\newcommand{\ccinterval}[2]{[#1,#2]}

\newcommand{\unit}{\ccinterval{0}{1}}

\newcommand{\tensor}{\otimes}

\newcommand{\NN}{\mathbb{N}}
\newcommand{\ZZ}{\mathbb{Z}}
\newcommand{\QQ}{\mathbb{Q}}
\newcommand{\RR}{\mathbb{R}}
\newcommand{\CC}{\mathbb{C}}
\newcommand{\HH}{\mathbb{H}}

\newcommand{\menge}[2]{\bigl\{ \thinspace #1 \thinspace\thinspace \big\vert%
\thinspace\thinspace #2 \thinspace \bigr\}}

\newtheorem{thm}{Theorem}
\newtheorem{lem}[thm]{Lemma}
\newtheorem{prop}[thm]{Proposition}
\newtheorem*{lem*}{Lemma}

\newtheorem*{thm*}{Theorem}

\theoremstyle{definition}

\theoremstyle{remark}

\theoremstyle{plain}

\DeclareMathOperator{\id}{id}

\renewcommand{\Im}{\operatorname{Im}}
\renewcommand{\Re}{\operatorname{Re}}

\newcommand{\shf}[1]{\mathscr{#1}}

\newcommand{\defeq}{\underset{\textrm{def}}{=}}

\newcommand{\argbl}{\_\!\_\!\_}

\def\overbar#1#2#3{{%
	\setbox0=\hbox{$#1$}%
	\dimen0=\wd0
	\advance\dimen0 by -#2 
	\vbox {\nointerlineskip \moveright #3 \vbox{\hrule height 0.3pt width \dimen0}%
		\nointerlineskip \vskip 1.5pt \box0}%
}}

\newcommand{\dst}{\Delta^{\!\ast}}
\newcommand{\dstn}[1]{(\smash[t]{\dst})^{#1}}
\newcommand{\CCstn}[1]{(\smash[t]{\CC^{\!\ast}})^{#1}}
\newcommand{\Mb}{\overbar{M}{4pt}{3pt}}
\newcommand{\Xb}{\overbar{X}{2.5pt}{2pt}}
\newcommand{\Tb}{\overbar{T}{2.5pt}{2pt}}
\newcommand{\TUb}{\overbar{T_U}{4pt}{2pt}}

\newcommand{\unitint}{\lbrack 0, 1 \rbrack}
\newcommand{\locsysH}{\shf{H}}
\DeclareMathOperator{\Aut}{Aut}

\newcommand{\shV}{\shf{V}}
\newcommand{\shO}{\shf{O}}

\newcommand{\shVb}{\overbar{\shV}{3pt}{3pt}}
\newcommand{\shOM}{\shf{O}_M}

\newcommand{\Chb}{\overline{C(h)}}

\theoremstyle{definition}
\newtheorem*{question}{Question}

\begin{document}

\title{Canonical extensions of local systems}
\author[C.~Schnell]{Christian Schnell}
\address{The Ohio State University\\
231 West 18th Avenue\\
Columbus, OH 43210}
\email{schnell@math.ohio-state.edu}


\begin{abstract}
A local system on a complex manifold $M$ can be viewed in two ways---either as
a locally free sheaf $\locsysH$ on $M$, or as a union of covering spaces
$T(\locsysH) \to M$. When $M$ is an open set in a bigger manifold $\Mb$,
the local system will generally not extend to $\Mb$, because of local monodromy.
This paper proposes an extension of the local system as an \emph{analytic space}
over $\Mb$, in the case when $\Mb \setminus M$ has normal crossing
singularities, and the local system is \emph{unipotent} along $\Mb
\setminus M$.  
The analytic space is obtained by taking the closure of $T(\locsysH)$ inside the
total space of Deligne's canonical extension of the vector bundle $\shOM \tensor
\locsysH$ to $\Mb$. It is not normal, but its normalization is locally toric.
\end{abstract}
\maketitle


\section{Introduction}

In his book about differential equations with regular singular points \cite{Deligne},
Deligne introduced the so-called ``canonical extension'' of a vector bundle with
flat connection. The most important case of his construction is the following. Say
$\shV$ is a holomorphic vector bundle on a complex manifold $M$, equipped with a flat
connection $\nabla$ (which means that $\nabla \circ \nabla = 0$). Now suppose that
$M$ is an open subset in a bigger complex manifold $\Mb$, in such a way that
\begin{enumerate}
\item the complement $\Mb \setminus M$ is a divisor with normal crossing singularities, and
\item the local monodromy of $\nabla$ near points of $\Mb \setminus M$ is unipotent.
\end{enumerate}
In this situation, Deligne shows that $\shV$ extends in a canonical manner to a
vector bundle $\shVb$ on $\Mb$, whose characteristic property is that in any local
frame for $\shVb$ near points of $\Mb \setminus M$, the connection matrix for
$\nabla$ has at worst logarithmic poles with nilpotent residues.\footnote{The
construction of $\shVb$, in local coordinates, is reviewed in the appendix.}

Local systems with integer coefficients are one source for flat vector bundles; if
$\locsysH$ is a local system of (finitely generated, free) abelian groups, then $\shV
= \shO_M \tensor_{\ZZ} \locsysH$, together with the natural connection, is such a
bundle. So if $\locsysH$ is unipotent along $\Mb \setminus M$, i.e., has unipotent
monodromy near points of $\Mb \setminus M$, then Deligne's construction applies, and
there is a canonical extension $\shVb$ for the vector bundle. It is then natural to
ask:

\begin{question} 
Does the original local system also extend in some way?
\end{question}

The present paper proposes an answer to this question. It should be clear that,
because of monodromy, $\locsysH$ cannot in general be extended to $\Mb$ as a local
system. We return therefore to the more old-fashioned view of a local system as a
\emph{space}, instead of as a sheaf---the total space $T = T(\locsysH)$ is a covering space
of $M$, typically infinite-sheeted and with countably many components. Now $T$ is
naturally embedded into the total space of the vector bundle $\shV$, and we shall
answer the question posed above by determining its closure $\Tb$ inside the total space of
the canonical extension $\shVb$.

Perhaps surprisingly, this closure is still an analytic space with good properties.
For instance, while $\Tb$ need not be normal, its normalization is locally toric, and
thus has controlled singularities, in a sense. This, and other things, will follow
from the explicit local equations for $\Tb$ that are given in
Section~\ref{sec:localequations}.

As discussed in Section~\ref{sec:universal}, the space $\Tb$ also has a sort of
universal property that might justify calling $\Tb$---or perhaps its
normalization---the ``canonical extension'' of the local system $\locsysH$. It is in
this sense that the title of the paper should be understood.

A precise description of $\Tb$ will be given below, so only two small issues need be
addressed here. One is that points of $\Tb \setminus T$ arise from
monodromy-invariant points in $\locsysH$. This is what one would expect;
indeed, $\Tb$ is also ``canonical'' in the sense that it includes all points that are
invariant under any part of the local monodromy near $\Mb \setminus M$. The other,
perhaps more surprising issue is that, even though $T$ itself is discrete over $M$,
the fibers of $\Tb$ over $\Mb$ need not be discrete.  In fact, they are unions of
affine spaces, of dimension possibly as big as $\dim M - 1$. 

\subsection*{Note}

The construction of $\Tb$ arose in a project that Herb Clemens and myself are working
on. Hoping that the result might be of independent interest, I am presenting it by
itself, without any of the original context. I warmly thank Herb Clemens for many
useful discussions and for his continuing help.


\section{Statement of the result}

This section gives a precise statement of the main result.  Let $M$ be a complex
manifold of dimension $n$, embedded as an open subset into a larger complex manifold
$\Mb$, in such a way that $\Mb \setminus M$ is a divisor with only normal crossing
singularities. Every point in $\Mb$ thus has a neighborhood isomorphic to $\Delta^n$,
with holomorphic coordinates $t_1, \dotsc, t_n$, in which the divisor $\Mb \setminus
M$ is defined by an equation of the form $t_1 \dotsm t_r = 0$.\footnote{Here, and in
the following, $\Delta \subseteq \CC$ denotes the open unit disk in the complex
plane, and $\dst = \Delta \setminus \{0\}$ the punctured unit disk.}

On $M$, we assume that we are given a local system $\locsysH$, with fiber $H \simeq
\ZZ^d$ a finitely generated free $\ZZ$-module. Up to isomorphism, it is determined by
the corresponding monodromy representation
\[
	\rho \colon G \to \Aut_{\ZZ}(H),
\]
where $G$ is the fundamental group of $M$ (for some choice of basepoint).

We shall assume that the local system $\locsysH$ is \emph{unipotent}; that is to say,
in a neighborhood $\Delta^n$ of each point, the fundamental group of $\Delta^n \cap
M$ should act by unipotent transformations on the fiber of $\locsysH$. It is then
possible to extend the holomorphic vector bundle $\shV = \locsysH \tensor \shO_M$ to
a vector bundle $\shVb$ on $\Mb$, using Deligne's construction. 

The total space $T = T(\locsysH)$ of the local system is naturally a subset of the
total space $T(\shVb)$ of this canonical extension. We are going to extend $T$ in the maximal
possible way, by taking its closure inside the total space of the vector
bundle. The resulting space is surprisingly nice, as witnessed by the following
theorem.

\begin{thm} \label{thm:Tb}
Let $\Tb$ be the closure of the total space $T$ of the local system, taken inside the
total space of the canonical extension of $\locsysH \tensor \shO_M$ over $\Mb$. Then
$\Tb$ has the following three properties:
\begin{enumerate}[label=(\roman{*}), ref=(\roman{*})]
\item $\Tb$ is a reduced analytic subset of $T(\shVb)$. \label{it:ASi}
\item The projection map $p \colon \Tb \to \Mb$ is holomorphic, 
	and $\finv[p](M) = T$. \label{it:ASii}
\item The normalization of $\Tb$ is locally toric. \label{it:ASiii}
\end{enumerate}
\end{thm}

\begin{proof}
To show that $\Tb$ is an analytic subset, we need to show that it is locally defined
by analytic equations inside the complex manifold $T(\shVb)$. For an arbitrary point
of $\Mb$, take a neighborhood isomorphic to $\Delta^n$ in which $\Mb \setminus M$ is
defined by the equation $t_1 \dotsm t_r = 0$. Then Proposition~\ref{prop:local} in
Section~\ref{sec:localequations} below states precisely that the closure of $T$ over
$\Delta^n$ is a reduced analytic subset, and this establishes \ref{it:ASi}.

The assertion in \ref{it:ASii} that $\finv[p](M) = T$ follows from the corresponding
statement over each coordinate neighborhood $\Delta^n$, also proved below (see the
discussion at the end of Section~\ref{sec:closure}). Finally, the statement about the
normalization of $\Tb$ may be found in Section~\ref{sec:singularities}.
\end{proof}


\section{The local situation} \label{sec:local}

Determining the closure $\Tb$ inside the total space of the canonical extension is
really a local problem, and we may restrict our attention to what happens in a small
polydisk neighborhood of a point $P \in \Mb$. Let $\Delta^n \subseteq \Mb$ be such a
neighborhood, with local holomorphic coordinates $t_1, \dotsc, t_n$ centered at the
point $P$ in question.  We assume that, in these coordinates, $\Mb \setminus M$ is
defined by the equation $t_1 \dotsm t_r = 0$.  When $P$ is a boundary point, we get
$r > 0$, but the case of a point in $M$ is included by taking $r = 0$. In any case,
we have $\dstn{n} \subseteq M$.

To avoid having to treat various cases based on the value of $r$, we will restrict
the local system $\locsysH$ to the set $\dstn{n}$, and compute the closure of only
this piece inside the total space of the canonical extension over $\Delta^n$. We
shall argue later, at the end of Section~\ref{sec:closure}, that the result is the
same.

The fundamental group of $\dstn{n}$ is isomorphic to $\ZZ^n$. Let $H \simeq \ZZ^d$ be
the fiber of the local system at some point in $\dstn{n}$; by assumption, the
monodromy action of $\ZZ^n$ on $\ZZ^d$ is by unipotent matrices. Let $T_j \in
\Aut_{\ZZ}(\ZZ^d)$ be the matrix corresponding to the $j$-th standard generator of
$\ZZ^n$, and put 
\[
	N_j = -\log T_j = \sum_{n = 1}^{\infty} \frac{1}{n} (\id - T_j)^n.
\]
This is well-defined because $(\id - T_j)^n = 0$ for large values of $n$. The
matrices $N_j$ are nilpotent, with rational entries, and commute with one another.  

In the given system of coordinates, we now describe how the local system is embedded
into the total space of the canonical extension. To begin with,
\[
	\HH^n \to \dstn{n}, \qquad (z_1, \dotsc, z_n) \mapsto 
		\bigl( e^{2\pi i z_1}, \dotsc, e^{2 \pi i z_n} \bigr),
\]
is the universal covering space of $\dstn{n}$. The canonical extension of $\locsysH
\tensor \shO_{\dstn{n}}$ over $\Delta^n$ is isomorphic to the trivial vector bundle 
\[
	\shO_{\Delta^n} s_1 \oplus \dotsb \oplus \shO_{\Delta^n} s_d;
\]
here $s_i$ is the section of $\locsysH \tensor \shO_{\dstn{n}}$ on $\dstn{n}$
whose pullback to $\HH^n$ is given by the map
\[
	\tilde{s}_i \colon \HH^n \to \ZZ^d, \qquad 
		\tilde{s}_i(z) = e^{z_1 N_1 + \dotsb + z_n N_n} e_i,
\]
$e_i$ being one of the standard basis elements of $\ZZ^d$ (see the appendix for
details). The total space of the canonical extension is thus isomorphic to $\Delta^n
\times \CC^d$, using this frame. 

When the local system is pulled back to the universal covering space $\HH^n$, it
becomes of course trivial. At any given point $z = (z_1, \dotsc, z_n)$ of $\HH^n$, a
class $h \in \ZZ^d$ in the fiber of the trivial local system has coordinates
$e^{-(z_1 N_1 + \dotsb + z_n N_n)} h$ with respect to the given framing for the
canonical extension. It follows that the point $(z, h) \in \HH^n \times \ZZ^d$
has coordinates
\[
	\bigl( e^{2\pi i z_1}, \dotsc, e^{2\pi i z_n}, 
		e^{-(z_1 N_1 + \dotsb + z_n N_n)} h \bigr)
\]
in $\Delta^n \times \CC^d$. The total space of the local system, when embedded into
that of the canonical extension, is thus the image of the holomorphic map
\[
	f \colon \HH^n \times \ZZ^d \to \Delta^n \times \CC^d,
\]
defined by the rule
\begin{equation} \label{eq:param}
	(z_1, \dotsc, z_n, h) \mapsto 
		\bigl( e^{2\pi i z_1}, \dotsc, e^{2\pi i z_n}, 
			e^{-(z_1 N_1 + \dotsb + z_n N_n)} h \bigr).
\end{equation}
The closure of this image will be computed in the following section.


\section{Description of the closure} \label{sec:closure}

We now determine which points inside the total space of the canonical extension
belong to the closure of the local system. As explained in the previous section, this
is a local question; we chose a neighborhood $\Delta^n \subseteq \Mb$ of an arbitrary
point $P \in \Mb$, and study the closure over that neighborhood. 

According to the description above, the total space of the local
system over $\dstn{n}$ is the image of the holomorphic map $f$ given in
\eqref{eq:param}. As it stands, that map is not one-to-one; when the real parts $x_j
= \Re z_j$ are restricted to $0 \leq x_1, \dotsc, x_n < 1$, however, every point in
the image is parametrized only once.

The remainder of this section is devoted to proving the following proposition, which
describes the points in the closure of the image of the map $f$.  As written, it only
makes a statement about points that lie over the origin in $\Delta^n$, which is to
say over the point $P \in \Mb$.  But as we are free to place $P$ wherever we please,
we really get a description of all the points in the closure. 

\begin{prop} \label{prop:closure}
A point in $\Delta^n \times \CC^d$ over $(0, \dotsc, 0) \in \Delta^n$ is in the
closure of the image of $f$ if, and only if, it is of the form
\[
	\bigl( 0, \dotsc, 0, e^{-(w_1 N_1 + \dotsb + w_n N_n)} h \bigr);
\]
here $h \in \ZZ^d$ is such that $a_1 N_1 h + \dotsb + a_n N_n h = 0$ for certain
positive integers $a_1, \dotsc, a_n$, while $w_1, \dotsc, w_n \in \CC$ can be arbitrary
complex numbers.

For each limit point, there is an arc (for suitably small $\eps > 0$)
\[
	\Delta(\eps) \to \Delta^n \times \CC^d,
\]
of the form
\[
	t \mapsto \bigl( t^{a_1} e^{2\pi i w_1}, \dotsc, t^{a_n} e^{2\pi i w_n}, 
		e^{-(w_1 N_1 + \dotsb + w_n N_n)} h \bigr),
\]
contained in the image of $f$ for $t \neq 0$, and passing through the limit point at $t = 0$.
\end{prop}

One half of this is easy to prove---if $h \in \ZZ^d$ satisfies $a_1 N_1 h + \dotsb +
a_n N_n h = 0$ for positive integers $a_1, \dotsc, a_n$, then every point of the
form
\[
	\bigl( 0, \dotsc, 0, e^{-(w_1 N_1 + \dotsb + w_n N_n)} h \bigr)
\]
is in the closure of the image of $f$. Indeed, taking the imaginary part of $z \in
\HH$ sufficiently large to have $\Im(a_j z + w_j) > 0$ for all $j$, we get
\begin{align*}
	f(a_1 z + w_1, \dotsc, a_n z + w_n, h) &= 
		\bigl( e^{2\pi i a_1 z} e^{2\pi i w_1}, \dotsc, e^{2\pi i a_n z} e^{2\pi i w_n},
			e^{-\sum (a_j z + w_j) N_j} h \bigr) \\ &=
		\bigl( t^{a_1} e^{2\pi i w_1}, \dotsc, t^{a_n} e^{2\pi i w_n},
			e^{-\sum w_j N_j} e^{-z \sum a_j N_j} h \bigr) \\ &=
		\bigl( t^{a_1} e^{2\pi i w_1}, \dotsc, t^{a_n} e^{2\pi i w_n}, 
			e^{-\sum w_j N_j} h \bigr),
\end{align*}
having set $t = \exp(2\pi i z)$. For $t \neq 0$, these points are all in the image of
$f$; as $t \to 0$, in other words, as $\Im z \to \infty$, they approach the point
$\bigl( 0, \dotsc, 0, e^{-(w_1 N_1 + \dotsb + w_n N_n)} h \bigr)$, which is consequently
in the closure.

To prove the converse, we take a sequence of points in the image that converges to
some point of $\{(0, \dotsc, 0)\} \times \CC^d$, and show that its limit is of the
stated form. So let
\[
	\bigl( z(m), h(m) \bigr) = \bigl( z_1(m), \dotsc, z_n(m), h(m) \bigr) \in
		\HH^n \times \ZZ^d
\]
be a sequence of points such that $f(z(m), h(m))$ converges to a point over $(0,
\dotsc, 0)$. This means that each $y_j(m) = \Im z_j(m)$ is going to infinity, and
that the sequence of vectors
\[
	e^{-\sum z_j(m) N_j} h(m) \in \CC^d
\]
is convergent as $m \to \infty$. Changing the values of $h(m)$, if
necessary, we may in addition assume that the real parts $x_j(m) = \Re z_j(m)$
satisfy $0 \leq x_j(m) \leq 1$.

In the course of the argument, we shall frequently have to pass to a subsequence of
$\bigl( z(m), h(m) \bigr)$. To avoid clutter, this will not be indicated in the
notation---in each case, the subsequence will be denoted by the same letters $\bigl(
z(m), h(m) \bigr)$ as the original sequence. Since it should not lead to any
confusion, we shall avail ourselves of this convenient device.

Keeping this convention in mind, we now proceed in several steps.

\subsection*{Step 1}

The sequence of real parts $x_j(m)$ is bounded, for each $j = 1, \dotsc, n$, and we
can thus pass to a subsequence where each $x_j(m)$ converges. The vectors
\[
	e^{\sum x_j(m) N_j} e^{- \sum z_j(m) N_j} h(m) = e^{- i \sum y_j(m) N_j} h(m)
\]
still form a convergent sequence in this case, and so the $x_j(m)$ really play no
role for the remainder of the argument.

\subsection*{Step 2}

While all imaginary parts $y_j(m)$ are going to infinity, this may happen at greatly
different rates. To make their behavior more tractable, we use the following
technique, borrowed from the paper by Cattani, Deligne, and Kaplan \cite{CDK}*{p.~494}. 
Let $y(m) = \bigl( y_1(m), \dotsc, y_n(m) \bigr)$. By taking a further subsequence,
we can arrange that
\[
	y(m) = \tau_1(m) \theta^1 + \dotsb + \tau_r(m) \theta^r + \eta(m),
\]
where $\theta^1, \dotsc, \theta^r \in \RR^n$ are constant vectors with nonnegative
components, and where the ratios
\begin{equation} \label{eq:tau}
	\tau_1(m) / \tau_2(m), \dotsc, \tau_{r-1}(m) / \tau_r(m), \tau_r(m)
\end{equation}
are all going to infinity. The remainder term $\eta(m)$, on the other hand, is
convergent. We can even assume that
\[
	0 \leq \theta_j^1 \leq \theta_j^2 \leq \dotsm \leq \theta_j^r
\]
for all $j$; because $y_j(m) \to \infty$, all components of the last vector $\theta^r$
have to be positive real numbers.

Now define
\[
	N(m) \defeq \sum_{j=1}^n \bigl( y_j(m) - \eta_j(m) \bigr) N_j.
\]
As in Step~1, the convergence of the expression $e^{-i \sum \eta_j(m) N_j}$ makes the
$\eta_j$ essentially irrelevant to the rest of the argument---the sequence
\[
	e^{-i N(m)} h(m)
\]
is still a convergent sequence.

\subsection*{Step 3}

For each multi-index $\alpha = (\alpha_1, \dotsc, \alpha_n) \in \NN^n$, we set
\[
	N^{\alpha} \defeq \prod_{j=1}^n N_j^{\alpha_j}.
\]
Since the $N_j$ are commuting nilpotent operators, $N^{\alpha} = 0$ whenever
$\abs{\alpha} = \alpha_1 + \dotsb + \alpha_n$ is sufficiently large.

We can thus let $p \geq 0$ be the smallest integer for which there is a subsequence
of $\bigl( z(m), h(m) \bigr)$ with
\[
	\text{$N^{\alpha} h(m) = 0$ 
			for all multi-indices $\alpha$ with $\abs{\alpha} \geq p + 1$.}
\]
Passing to this subsequence, we find that when $\abs{\alpha} = p$, the sequence
\[
	N^{\alpha} e^{-i \sum y_j(m) N_j} h(m) = N^{\alpha} h(m)
\]
is convergent. However, it takes its values in a discrete set (in fact, there is an
integer $M > 0$ such that each coordinate of $N^{\alpha} h(m)$ is in $\ZZ \lbrack 1/M
\rbrack$, and $M$ depends only on $\alpha$ and the $N_j$), and so it has to be
eventually constant. If we remove finitely many terms from the sequence, we can
therefore achieve that
\[
	h^{\alpha} \defeq N^{\alpha} h(m)
\]
is constant whenever $\abs{\alpha} = p$. Moreover, we have $N(m) h^{\alpha} = 0$ by
the choice of $p$.

\subsection*{Step 4} 

At this point, we can use an inductive argument to get the conclusion of Step~3 for
all multi-indices $\alpha$ with $\abs{\alpha} \leq p$. Thus let us assume that we
already have a subsequence $\bigl( z(m), h(m) \bigr)$ for which $h^{\alpha} =
N^{\alpha} h(m)$ is constant and $N(m) h^{\alpha} = 0$, whenever $\alpha$ is a
multi-index with $p' \leq \abs{\alpha} \leq p$.  If $p' > 0$, we now show how to get
the same statement with $p'$ replaced by $p' - 1$.

Consider a multi-index $\alpha$ with $\abs{\alpha} = p' - 1$. Then
\begin{equation} \label{eq:indstep}
\begin{split}
	N^{\alpha} e^{-i N(m)} h(m) = 
	N^{\alpha} h(m) &- i N(m) N^{\alpha} h(m) \\
		&+ \sum_{s=1}^{p-p'} (-i)^{s+1} N(m)^s \cdot N(m) N^{\alpha} h(m)
\end{split}
\end{equation}
is again convergent. Since $\alpha + e_j$ has length $p'$, we see that
\[
	N(m) N^{\alpha} h(m)
		= \sum_{j=1}^n \bigl( y_j(m) - \eta_j(m) \bigr)  N^{\alpha + e_j} h(m) 
		= \sum_{j=1}^n \bigl( y_j(m) - \eta_j(m) \bigr) h^{\alpha + e_j};
\]
by the inductive hypothesis, the last term in \eqref{eq:indstep} is actually zero. 

Thus the sequence $N^{\alpha} h(m) - i N(m) N^{\alpha} h(m)$
is itself convergent, implying convergence of its real and imaginary parts
separately. As before, the sequence of real parts $N^{\alpha} h(m)$ has to be eventually
constant, and after omitting finitely many terms, we can assume that it is constant. Let
\[
	h^{\alpha} \defeq N^{\alpha} h(m)
\]
be that constant value. Then the convergence of the imaginary part
\[
	N(m) h^{\alpha} = N(m) N^{\alpha} h(m) 
		= \sum_{i=1}^r \tau_i(m) \sum_{j=1}^n \theta_j^i N^{\alpha + e_j} h(m) 
		= \sum_{i=1}^r \tau_i(m) \sum_{j=1}^n \theta_j^i h^{\alpha + e_j},
\]
together with the behavior of the $\tau_i(m)$ described in \eqref{eq:tau}, shows that 
\[
	\sum_{j=1}^n \theta_j^i h^{\alpha + e_j} = 0
\]
for all $i$. But this says that, in fact, $N(m) h^{\alpha} = 0$. The statement is
thus proved for all multi-indices $\alpha$ of length $\abs{\alpha} = p' - 1$ as well.

\subsection*{Step 5}

From Step~4, we conclude that, on a suitable subsequence, $h^{\alpha} = N^{\alpha}
h(m)$ is constant for all $\alpha$, and satisfies $N(m) h^{\alpha} = 0$. In
particular, $h(m)$ is itself constant, equal to a certain element $h = h^{(0, \dotsc,
0)} \in \ZZ^d$. Moreover, we have $N(m) h = 0$ for all $m$. 

On the one hand, we now find that, along the subsequence we have chosen in the
previous steps, our original convergent sequence simplifies to
\[
	e^{-\sum z_j(m) N_j} h(m) 
		= e^{-\sum (x_j(m) + i \eta_j(m)) N_j} e^{-i N(m)} h 
		= e^{-\sum (x_j(m) + i \eta_j(m)) N_j} h.
\]
If we set $w_j = \lim_{m \to \infty} \bigl( x_j(m) + i \eta_j(m) \bigr)$, then the
limit of the sequence is of the form $e^{-\sum w_j N_j} h$, which was part of the
assertion in Proposition~\ref{prop:closure}.

On the other hand, we conclude from
\[
	N(m) h = \sum_{i=1}^r \tau_i(m) \sum_{j=1}^n \theta_j^i N_j h = 0
\]
that 
\[
	\sum_{j=1}^n \theta_j^i N_j h = 0
\]
for all $i = 1, \dotsc, r$.

\subsection*{Step 6}

By Step~5, we know that the $n$ vectors $N_j h$ are linearly dependent; the
coefficients $\theta_j^r$ in the relation (for $i = r$) are positive real numbers.
But as the vectors themselves are in fact in $\QQ^d$, we can also find a relation
with positive rational coefficients. Taking a suitable multiple, we then obtain
positive integers $a_1, \dotsc, a_n$ satisfying
\[
	\sum_{j=1}^n a_j N_j h = 0.
\]
The remaining assertion of the proposition is thereby established, and this finishes
the proof.

For later use, we record the result of the six steps in the following proposition.

\begin{prop} \label{prop:6steps}
Let $\bigl( z(m), h(m) \bigr) \in \HH^n \times \ZZ^d$ be a sequence of points with
$x_j(m) = \Re z_j(m) \in \unit$, and assume that $f\bigl(z(m), h(m)\bigr)$ converges
to a point
in $\Delta^n \times \CC^d$ over $(0, \dotsc, 0) \in \Delta^n$. Then there is a
subsequence, still denoted $\bigl( z(m), h(m) \bigr)$, for which $h(m)$ is constant.
\end{prop}

\subsection*{A technical point}

For the sake of convenience, we had restricted the local system from the open set $U
= \Delta^n \cap M$---whose complement in $\Delta^n$ is the normal crossing
divisor with equation $t_1 \dotsm t_r = 0$---to $\dstn{n}$, and computed the closure
of only this smaller piece. We now have to argue that this makes no difference.

So let $T_U \subseteq U \times \CC^d$ be the total space of the local system over
$U$; each connected component of $T_U$ is a covering space of $U$, and therefore the
subset $T_U \cap \bigl( \dstn{n} \times \CC^d \bigr)$ is itself dense in $T_U$. But
from this circumstance, it follows immediately that it has the same closure in
$\Delta^n \times \CC^d$ as $T_U$. Since $T_U$ is already closed inside of $U \times
\CC^d$, we see in particular that all points of $\TUb \setminus T_U$ have to lie over
the boundary divisior $t_1 \dotsm t_t = 0$.

(This last fact can also be inferred from Proposition~\ref{prop:closure}. For if we
take $P$ to be a point in $M$, the local system is already defined at $P$, and
trivial on a neighborhood $\Delta^n$ of that point. Consequently, the local monodromy
over $\dstn{n}$ is also trivial, and Proposition~\ref{prop:closure} shows that taking
the closure is only adding back a copy of $\ZZ^d$ over $P$. This means that we get
back the original fiber $\locsysH_P$ of the local system.)


\section{Local equations} \label{sec:localequations}

Now that we know which points are in the closure, we need to show that $\Tb$ is 
an analytic space. We shall do this by finding explicit local equations, over the
same coordinate neighborhoods that were used in the previous section.  Consequently,
$\Delta^n \subseteq \Mb$ will continue to denote a neighborhood of an arbitrary
point in $\Mb$, with local holomorphic coordinates $t_1, \dotsc, t_n$.

It has already been pointed out that the total space of the local system over $M$ has
countably many connected components. Locally, over the much smaller open set
$\Delta^n$, those components may break up even further. The map
\[
	f \colon \HH^n \times \ZZ^d \to \Delta^n \times \CC^d,
\]
parametrizing the total space of the local system over $\Delta^n$, was defined above
by the rule
\[
	(z_1, \dotsc, z_n, h) \mapsto 
		\bigl( e^{2\pi i z_1}, \dotsc, e^{2\pi i z_n}, 
				e^{-(z_1 N_1 + \dotsb + z_n N_n)} h \bigr).
\]
If we take any element $h \in H = \ZZ^d$, the image of
\[
	f(\argbl, h) \colon \HH^n \to \Delta^n \times \CC^d,
\]
denoted by $C(h)$, is one of the \emph{local} connected components of the total
space. Obviously, two such components $C(h_0)$ and $C(h_1)$ are either the same
(which is the case if $h_0$ and $h_1$ are in the same $\ZZ^n$-orbit), or disjoint;
this means that, typically, each component is equal to $C(h)$ for infinitely many $h
\in H$.

In the previous section, we have described the closure of
\[
	\bigcup_{h \in H} C(h)
\]
inside of $\Delta^n \times \CC^d$, but only as a set. To show that this closure is
actually an analytic space, we need to give holomorphic equations that define it
inside $\Delta^n \times \CC^d$. We first observe that, as a matter of fact, 
\[
	\overline{\bigcup\limits_{h \in H} C(h)} = \bigcup_{h \in H} \Chb.
\]
This is because of the description of the closure given in
Proposition~\ref{prop:closure}---any point in the closure is already in the closure
of one of the components $C(h)$.

Next, as would be expected if the closure is an analytic space, only finitely many of
the $\Chb$ can come together at the boundary. This is expressed in the following
lemma.

\begin{lem} \label{lem:fincomps}
At most finitely many distinct $\Chb$ can meet at any given point in $\Delta^n \times \CC^d$. 
\end{lem}

\begin{proof}
Suppose, to the contrary, that infinitely many distinct $\Chb$ met at a certain point
$Q$ of the closure. Such a point $Q$ necessarily lies in the closure of infinitely
many distinct sheets $C(h)$. Moving the center $P$ of the coordinate system, if
necessary, we may assume that $Q$ is a point over $(0, \dotsc, 0) \in \Delta^n$. We
can then find a sequence of points
$\bigl( z(m), h(m) \bigr) \in \HH^n \times \ZZ^d$, with $0 \leq \Re z_j(m) \leq 1$
for all $j = 1, \dotsc, n$, such that $f\bigl(z(m), h(m)\bigr)$ converges to $Q$, but all
$h(m)$ are distinct. But such a sequence cannot exist by
Proposition~\ref{prop:6steps}. This contradiction proves that the number of
components meeting at $Q$ is indeed finite.
\end{proof}

We are now ready to determine equations for each of the closed subsets $\Chb$ of
$\Delta^n \times \CC^d$. Let $h \in H$ be an arbitrary element; we will give finitely
many holomorphic equations defining $\Chb$. As before, we break the argument down
into several steps.

\subsection*{Step 1} 

According to Proposition~\ref{prop:closure}, we get additional points in the closure
$\Chb$ only when $h$ is invariant under some part of the monodromy action.  Thus we
let $S \subseteq \ZZ^n$ be the subgroup of elements that leave $h$ invariant.  As a
subgroup of a free group, $S$ is itself free, say of rank $n-k$. If $k = n$, then
$C(h)$ is already closed; only when $k < n$ is the closure $\Chb$ strictly bigger
than $C(h)$. Since the first case is essentially trivial, we shall assume from now on
that $k < n$.

\subsection*{Step 2}
The quotient $\ZZ^n / S$ is a free abelian group.

\begin{proof}
Since $\ZZ^n$ acts on $H = \ZZ^d$ by unipotent transformations, we have that
\[
	T_1^{a_1} \dotsm T_n^{a_n} h = h \qquad \text{if, and only if,} \qquad
		a_1 N_1 h + \dotsb + a_n N_n h = 0.
\]
This means that $S$ is the kernel of the homomorphism
\[
	\ZZ^n \to \QQ^d, \qquad 
		(a_1, \dotsc, a_n) \mapsto a_1 N_1 h + \dotsb + a_n N_n h,
\]
and so the quotient $\ZZ^n / S$ embeds into $\QQ^d$. Since $\QQ^d$ is torsion-free,
the same has to be true for $\ZZ^n / S$; being finitely generated, $\ZZ^n / S$
is then actually free.
\end{proof}

\subsection*{Step 3}

Because of Step~2, we can now find an $n \times n$ matrix $A$, with
integer entries and $\det A = 1$, whose last $n-k$ columns give a basis for the
subgroup $S$. We then introduce new coordinates $(w_1, \dotsc, w_n) \in \CC^n$ by the
rule
\begin{equation} \label{eq:zcoord}
	z_i \defeq \sum_{j = 1}^n a_{i,j} w_j.
\end{equation}
Rewriting $z_1 N_1 + \dotsb z_n N_n$ in the form $w_1 M_1 + \dotsb + w_n M_n$, where
each
\[
	M_j \defeq \sum_{i = 1}^n a_{i,j} N_i
\]
is still nilpotent, we now have $M_{k+1} h = \dotsb = M_n h = 0$, while
the remaining $k$ vectors $M_1 h, \dotsc, M_k h$ are linearly independent.
Instead of $f$, we can then use the parametrization
\begin{equation} \label{eq:newparam}
	g \colon V \to \Delta^n \times \CC^d, \qquad
		(w_1, \dotsc, w_n) \mapsto \bigl( t_1, \dotsc, t_n, e^{-(w_1 M_1 + \dotsb
			+ w_k M_k)} h \bigr),
\end{equation}
of the sheet $C(h)$ under consideration; here
\[
	t_j = \prod_{s = 1}^n e^{2\pi i a_{j,s} w_s},
\]
and the map $g$ is defined on the open subset $V \subseteq \CC^n$ where all 
$\abs{t_j} < 1$.

\subsection*{Step 4}

We now analyze the term $e^{-(w_1 M_1 + \dotsb + w_k M_k)} h$ in the parametrization
$g$. As a matter of fact, the map
\[
	\CC^k \to \CC^d, \qquad (w_1, \dotsc, w_k) \mapsto 
		e^{-(w_1 M_1 + \dotsb + w_k M_k)} h,
\]
is a closed embedding, because the vectors $M_1 h, \dotsc, M_k h$ are linearly
independent. We will prove this by constructing an inverse---we show that there are
polynomials $p_1(v), \dotsc, p_k(v)$ in $v = (v_1, \dotsc, v_d)$, 
such that whenever $v$ is in the image, one has
\[
	(w_1, \dotsc, w_k) = \bigl( p_1(v), \dotsc, p_k(v) \bigr).
\]

\begin{proof}
We construct suitable polynomials by induction on the number $k$ of variables. If $k
= 0$, there is nothing to do. So let us assume that the existence of such polynomials is
known for $k - 1 \geq 0$ variables, and let us establish it for $k$.

For any multi-index $\alpha = (\alpha_1, \dotsc, \alpha_k) \in \NN^k$, we write
\[
	M^{\alpha} \defeq M_1^{\alpha_1} \dotsm M_{k}^{\alpha_k};
\]
these matrices are zero whenever $\abs{\alpha}$ is sufficiently large.
Among all multi-indices $\alpha$ with $M^{\alpha} h \neq 0$,
select one of maximal length $\abs{\alpha}$. Then $\abs{\alpha} \geq 1$, because the
vectors $M_j h$ are in particular nonzero, and without loss of generality we may
assume that $\alpha_k \geq 1$.  We have
\[
	M^{\alpha - e_k} v = 
		\bigl( \id - w_1 M_1 - \dotsb - w_{k-1} M_{k-1} \bigr) M^{\alpha - e_k} h 
			- w_k M^{\alpha} h.
\]
Because at least one of the components of $M^{\alpha} h$ is non-zero, we can now solve
for $w_k$ in the form
\[
	w_k = c_1 w_1 + \dotsb + c_{k-1} w_{k-1} + l(v),
\]
with $c_1, \dotsc, c_{k-1} \in \QQ$, and $l(v)$ a degree-one polynomial in $v$.
Substituting back, we obtain
\[
	e^{l(v) M_k} v 
		= e^{- w_1 (M_1 + c_1 M_k) - \dotsb - w_{k-1} (M_{k-1} + c_{k-1} M_k) } h,
\]
and, by the inductive hypothesis, $w_1, \dotsc, w_{k-1}$ can now be expressed as
polynomials in the coordinates of the vector $e^{l(v) M_k} v$, since the vectors
$(M_i + c_i M_k)h$ are of course still linearly independent. It is thus possible
to find polynomials in $v$ such that
\[
	(w_1, \dotsc, w_{k-1}) = \bigl( p_1(v), \dotsc, p_{k-1}(v) \bigr).
\]
Then $w_k = c_1 p_1(v) + \dotsb + c_{k-1} p_{k-1}(v) + l(v)$ is also a polynomial in
$v$, and the assertion is proved.
\end{proof}

\subsection*{Step 5}

The result of Step~4 now gives us half of the equations for the closed subset $\Chb$.
Indeed, we have seen that if $(t, v) \in \Delta^n \times \CC^d$ is a point of $C(h)$,
then it is in the image of $g$, and so its $v$-coordinates satisfy the relation
\begin{equation*} \label{eq:A}
	v = e^{-(p_1(v) M_1 + \dotsb + p_k(v) M_k)} h.
		\tag{A}
\end{equation*}
In components, these are $d$ polynomial equations for $v = (v_1, \dotsc, v_d)$. The
same equations obviously have to hold for every point in the closure $\Chb$.

\subsection*{Step 6}

Next, we turn our attention to the remaining $n$ coordinates $(t_1, \dotsc, t_n)$ of
the parametrization $g$. Each is of the form
\[
	t_j = \prod_{s = 1}^n e^{2\pi i a_{j,s} w_s}.
\]
Letting $u_j = \exp(2 \pi i w_{j})$, for $j = k+1, \dotsc, n$, we have
\[
	t_j = u_{k+1}^{a_{j,k+1}} \dotsm u_{n\vphantom{k+1}}^{a_{j\vphantom{k+1},n}} \cdot 
		e^{2\pi i (a_{j,1} w_1 + \dotsb + a_{j,k} w_k)}.
\]
The shape of these products leads us to consider the algebraic map
\begin{equation} \label{eq:toric}
	\CCstn{n-k} \to \CC^n, \qquad (u_{k+1}, \dotsc, u_n) \mapsto (x_1, \dotsc, x_n),
\end{equation} 
whose coordinates are given by
\begin{equation} \label{eq:algmap}
	x_j = \prod_{i=k+1}^n u_i^{a_{j, i}}.
\end{equation}
Because the map is given by polynomials, the (topological) closure of its image is
actually a closed algebraic subvariety of $\CC^n$, and as such defined by finitely
many polynomial equations
\[
	f_1(x_1, \dotsc, x_n) = \dotsb = f_e(x_1, \dotsc, x_n) = 0.
\]
In fact, because the original map is monomial, each $f_b(x)$ can be taken as a
binomial in the variables $x_1, \dotsc, x_n$. (We shall have to say more later about
the toric structure of the image.)

\subsection*{Step 7}

From Step~6, we can now deduce the remaining equations for $\Chb$. Indeed, a point
$(t,v)$ in the image of $g$ has to satisfy the equations
\[
	f_b \Bigl( t_1 e^{-2\pi i \sum_{s \leq k} a_{1,s} w_s}, \dotsc, t_n e^{-2\pi i
		\sum_{s \leq k} a_{n,s} w_s} \Bigr) = 0
\]
for $b = 1, \dotsc, e$. From Step~4 we know, moreover, that $w_s = p_s(v)$; therefore
\begin{equation*} \label{eq:B}
	f_b \Bigl( t_1 e^{-2\pi i \sum_{s \leq k} a_{1,s} p_s(v)}, \dotsc, t_n e^{-2\pi i
		\sum_{s \leq k} a_{n,s} p_s(v)} \Bigr) = 0
		\tag{B}
\end{equation*}
is another set of $e$ holomorphic equations satisfied by the closure $\Chb$.

\subsection*{Step 8}

It remains to see that the $d+e$ equations in \eqref{eq:A} and \eqref{eq:B}
really define $\Chb$, and not a bigger set. The trivial case is when $h$ is not
invariant under any part of the monodromy; here $k = n$, and as pointed out in
Step~1, $C(h)$ is then already a closed set, and there is nothing to show. In the
remaining case, when $k < n$, we are free to place the point $P$ (the center of our
coordinate system $t_1, \dotsc, t_n$) anywhere we like. A moment's thought shows
that it therefore suffices to consider solutions of the equations over $(0, \dotsc,
0) \in \Delta^n$, and to prove that those have to lie in the closure of $C(h)$.

So consider a point $(0,v) \in \Delta^n \times \CC^d$ that satisfies the equations.
On the one hand, the equations in \eqref{eq:A} define the image of a closed
embedding, as explained in Step~4; therefore, $v = e^{-(w_1 M_1 + \dotsb + w_k M_k)}
h$ for a unique point $(w_1, \dotsc, w_k) \in \CC^k$. Letting $w = (w_1, \dotsc, w_k,
0, \dotsc, 0)$, and going back to the original coordinates $z$ in \eqref{eq:zcoord},
we get a point $(z_1, \dotsc, z_n) \in \CC^n$ such that
\[
	v = e^{-(z_1 N_1 + \dotsb + z_n N_n)} h.
\]

On the other hand, the equations in \eqref{eq:B} arose from the 
map defined in Step~6. Now \eqref{eq:algmap} shows that the point $(0, \dotsc, 0)$
can only be in the closure of the image when some linear combination of the
exponent vectors $(a_{1,i}, \dotsc, a_{n,i})$, for $i = k+1, \dotsc, n$, has positive
coordinates. Since these vectors generate the subgroup $S$, we thus get positive
integers $a_1, \dotsc, a_n$ with
\[
	a_1 N_1 h + \dotsb + a_n N_n h = 0
\]
But by the description in Proposition~\ref{prop:closure}, this says exactly that the
point $(0,v)$ belongs to $\Chb$.

In summary, we have established the following two results. First, we can give local
equations for each of the components $\Chb$.

\begin{prop} \label{prop:localsheet}
The closure $\Chb$ of the sheet $C(h)$ in $\Delta^n \times \CC^d$ is an analytic
subset, defined by $d + e$ holomorphic equations in the coordinates $(t_1, \dotsc,
t_n, v_1, \dotsc, v_d)$. These equations are, firstly,
\[
	v = e^{-(p_1(v) M_1 + \dotsb + p_k(v) M_k)} h.
\]
and, secondly,
\[
	f_b \Bigl( t_1 e^{-2\pi i \sum_{s \leq k} a_{1,s} p_s(v)}, \dotsc, t_n e^{-2\pi i
		\sum_{s \leq k} a_{n,s} p_s(v)} \Bigr) = 0
\]
for $b = 1, \dotsc, e$. In particular, as the closure of the complex submanifold
$C(h)$, the subset $\Chb$ is itself a reduced and irreducible analytic space.
\end{prop}

Moreover, because only finitely many components can meet at any given point (by
Lemma~\ref{lem:fincomps}), we can conclude that the closure of the image of $f$ is an
analytic subset of $\Delta^n \times \CC^d$ as well.

\begin{prop} \label{prop:local}
Let $T_U$ be the total space of the local system over the open set $U = M \cap
\Delta^n$. The closure of $T_U$ inside of $\Delta^n \times \CC^d$ is a reduced analytic
subset with countably many irreducible components, each of the form $\Chb$ for some
$h \in \ZZ^d$.
\end{prop}


\section{Singularities} \label{sec:singularities}

The analytic space $\Tb$, described in Theorem~\ref{thm:Tb}, will in general be
singular at points not in $T$. This is apparent from the discussion in the previous
section---on the one hand, several of the local components $C(h)$ may be coming
together at the boundary (see Lemma~\ref{lem:fincomps}); on the other hand, the local
equations of the closure are such that singularities have to be expected even for
each $C(h)$ itself (see Steps~6 and 7 in the previous section). There are two
possible approaches to this problem---normalization, and resolution of singularities.

\subsection*{Normalization}

For various applications, it is desirable to have at least a normal space; mostly
because it is then possible to extend holomorphic maps that are naturally defined on
$T$ to all of $\Tb$, by showing that they extend in codimension one. In addition, the
normalization of $\Tb$ is an unexpectedly nice space.

According to the local description of $\Tb$ given above, the process of normalizing
$\Tb$ has two effects. Firstly, it separates all the local components $\Chb$ at points
where they meet, making them disjoint. Secondly, it normalizes each $\Chb$ itself.
From Step~6 in the previous section, we see that $\Chb$ is locally isomorphic to a
(typically non-normal) toric variety. Indeed, the map $g$ in \eqref{eq:newparam},
whose image is the sheet $C(h)$, is locally the product of a closed immersion and a
map defined by monomials. As explained in the article by David Cox
\cite{Cox}*{p.~402}, the closure of the image of a monomial map as in
\eqref{eq:toric} is a non-normal toric variety, and after taking the normalization,
one gets a toric variety in the usual sense.  It follows that the normalization of
each $\Chb$ is locally toric; in particular, this means that it has only rational
singularities. The same is therefore true for the normalization of $\Tb$ itself.

\subsection*{Resolution of singularities}

A second possibility is to resolve all the singularities of $\Tb$. By construction,
the total space $T$ of the local system is a nonsingular dense open subset of $\Tb$.
Its complement $\Tb \setminus T$, being the preimage of the divisor $\Mb \setminus M$
under the holomorphic projection map from $\Tb$ to $\Mb$, is a closed analytic
subspace. According to the results of Bierstone and Milman \cite{BM}*{p.~298}, it is
possible to resolve the singularities of $\Tb$ by blowing up, at the same time making
the preimage of $\Tb \setminus T$ into a divisor with only normal crossing
singularities. Since the centers of the blowups can be chosen to lie outside of $T$,
the resulting complex manifold will still have $T$ as a dense open subset. Of course,
the space one gets is as ``canonical'' as the resolution process. 

Since the normalization of $\Tb$ is locally toric, one can also normalize first, and
then use the older results on desingularizing toroidal embeddings
\cite{KKMSD}*{p.~94} to create a non-singular space from $\Tb$.

\section{A universal property of the extension} \label{sec:universal}

In this section, we shall give some justification for calling the space $\Tb$ the
``canonical extension'' of the local system $\locsysH$. The following proposition is
the main result in this direction.

\begin{prop} \label{prop:universal}
Let $g \colon \Xb \to \Mb$ be an arbitrary holomorphic map from a reduced and \emph{normal}
analytic space $\Xb$ to $\Mb$. Assume that the open set $X = \finv[g](M)$ is dense in
$\Xb$, and that there is a factorization
\[
\begindc{\commdiag}[25]
\obj(1,1)[a]{$X$}
\obj(3,1)[b]{$M$}
\obj(3,3)[c]{$T$}
\mor{a}{b}{$g$}
\mor{a}{c}{$s$}[\atleft,\dasharrow]
\mor{c}{b}{$p$}
\enddc
\]
as in the diagram. Then $s$ extends uniquely to a holomorphic map $s \colon \Xb \to
\Tb$, making 
\[
\begindc{\commdiag}[25]
\obj(1,1)[a]{$\Xb$}
\obj(3,1)[b]{$\Mb$}
\obj(3,3)[c]{$\Tb$}
\mor{a}{b}{$g$}
\mor{a}{c}{$s$}[\atleft,\dasharrow]
\mor{c}{b}{$p$}
\enddc
\]
commute.
\end{prop} 

In order to prove this, we shall first reformulate the statement. Let $\shVb$ be the
canonical extension of the flat vector bundle $\shV = \shO_M \tensor \locsysH$ to
$\Mb$; then $\Tb$ is the closure of $T$ inside the total space of $\shVb$. The map $s
\colon X \to T$ gives a section of the pullback of the local system $\finv[g]
\locsysH$---as well as of the bundle $\fust[g] \shVb$---over $X$, that we continue to
denote by $s$.  Now the statement of the proposition is equivalent to saying that $s$
extends to a section of $\fust[g] \shVb$ over $\Xb$. Indeed, such an extension
(clearly unique if it exists) gives a map from $\Xb$ to the total space of $\shVb$,
and since $X$ is mapped into $T$, the image has to be contained in the closure $\Tb$.
Thus extending the map $s$ to $\Xb$ is equivalent to extending the corresponding
section $s$ of $\fust[g] \shVb$.

We now begin the proof by establishing the following special case.

\begin{lem} 
The conclusion of Proposition~\ref{prop:universal} holds whenever $\Xb = \Delta$ is
the unit disk, and $X = \dst$.
\end{lem}

\begin{proof}
Let $g \colon \Delta \to \Mb$ be the given morphism. Since the original local system
$\locsysH$ is unipotent along $\Mb \setminus M$, its pullback $\finv[g] \locsysH$ to
$\dst$ has unipotent monodromy around $0 \in \Delta$. Thus the vector bundle
$\fust[g] \shVb$ is the canonical extension of $\shO_{\dst} \tensor \finv[g]
\locsysH$ to $\Delta$. 

As we said, it suffices to show that the section $s \in H^0(\dst, \fust[g] \shVb)$
extends to $\Delta$; but this follows very easily from the construction of the
canonical extension. Indeed, if we let $N$ be the logarithm of the monodromy of
$\finv[g] \locsysH$ on $\dst$, then the description on p.~\pageref{sec:local} shows
that the total space of $\finv[g] \locsysH$, inside that of the bundle $\fust[g]
\shVb$, is given by the image of the map
\[
	\HH \times \ZZ^d \to \dst \times \CC^d, \qquad
	(z, h) \mapsto \bigl( e^{2\pi i z}, e^{-zN}h \bigr),
\]
in a suitable frame of $\fust[g] \shVb$. The section $s$ corresponds to a
monodromy-invariant element $h \in \ZZ^d$, satisfying $N h = 0$, and so the whole
sheet $\Delta \times \{h\}$ lies in the closure of the image. This shows that $s$
extends over $0$, proving the lemma.
\end{proof}

We can now turn to proving Proposition~\ref{prop:universal} in general. Let $g \colon
\Xb \to \Mb$ be the given map, and set $Z = \Xb \setminus X$. As before, $s \in
H^0(X, \fust[g] \shVb)$ denotes the section of the pullback bundle corresponding to
the given factorization. In order to show that $s$ extends holomorphically to all of
$\Xb$, it suffices to show that it extends in codimension one, since $\Xb$ is normal
(see \cite{Narasimhan}*{p.~118}, for example). The singular locus of $\Xb$ has
codimension at least two; it is therefore enough to prove that $s$ extends across
those points $P \in Z$ where both $Z$ and $\Xb$ are nonsingular, and
$\mathop{codim}_P(Z,X) = 1$.

This is a local question, and after choosing suitable local coordinates $z_1, \dotsc,
z_m$ on a small neighborhood $U$ of $P$ in $\Xb$, we can assume that $Z \cap U$ is
defined by the equation $z_m = 0$. Applying the lemma to maps of the form \[ \Delta
\to U, \qquad t \mapsto (z_1, \dotsc, z_{m-1}, t), \] we see that the section $s$
extends across $P$. This completes the proof of the proposition.

\subsection*{Note}

Since $\Tb$ itself is usually not a normal space, Proposition~\ref{prop:universal} is
not quite as strong as one might like it to be. It is, however, easy to make examples
of maps from non-normal spaces $\Xb$ to $\Mb$ (for instance, taking $\Xb$ to be a
nodal curve) where the statement is false. This suggests that the normalization of
$\Tb$ is the space that deserves to be called the ``canonical extension'' of the
local system $\locsysH$.


\appendix
\section*{Conventions} \label{sec:conventions}

This short section lists various conventions that are used throughout the paper. The
construction of Deligne's canonical extension is also reviewed, in a form suitable
for our proof.

\subsection*{Fundamental group}

If $X$ is a topological space, with basepoint $x \in X$, we write $\pi_1(X, x)$ for
its fundamental group. Given two closed paths $\gamma, \delta \colon \unitint \to X$,
with $\gamma(0) = \gamma(1) = \delta(0) = \delta(1) = x$, representing two elements
of $\pi_1(X, x)$, their product is defined as
\[
	\gamma \delta \colon \unitint \to X, \qquad 
	t \mapsto 
		\begin{cases}
			\; \gamma(2t) \quad &\text{if $0 \leq t \leq 1/2$,} \\
			\; \delta(2t-1) \quad &\text{if $1/2 \leq t \leq 1$.}
	\end{cases}
\]

Writing $\Delta$ for the unit disk in $\CC$, and $\dst = \Delta \setminus \{0\}$ for
the punctured disk, we shall take the generator of $\pi_1(\dst) \simeq \ZZ$ to be a
loop that goes around the origin once, counter-clockwise.

\subsection*{Action on the fiber} 

For any covering space $p \colon Y \to X$ of $X$, the group $\pi_1(X, x)$ acts on the
fiber $\finv[p](x)$ by a left action; given $\gamma \in \pi_1(X, x)$ and a point $y
\in Y$ with $p(y) = x$, one has $\gamma \cdot y = \tilde{\gamma}(0)$, where $\gamma$
has been lifted to a path $\tilde{\gamma} \colon \unitint \to Y$ with
$\tilde{\gamma}(1) = y$. Put more succinctly, $\gamma$ acts by parallel translation along
the path $\gamma^{-1}$.

\subsection*{Local systems}

A special case of this is the correspondence between local systems on $X$ and
representations of the fundamental group. Given a local system $\locsysH$ of abelian
groups on $X$, say with fiber $H = \locsysH_{x}$, each connected component of the
total space of $\locsysH$ is a covering space of $X$. One obtains a representation 
\[
	\rho \colon \pi_1(X, x) \to \Aut(H)
\]
by letting the fundamental group act on the fiber. From the representation $\rho$, on
the other hand, one can recover $\locsysH$. Indeed, if $p \colon \tilde{X} \to X$ is the
universal covering space of $X$ (assuming its existence), the quotient of $\tilde{X}
\times H$ by the group action
\[
	\gamma \cdot \bigl( \tilde{x}, h \bigr) = 
		\bigl( \gamma \cdot \tilde{x}, \rho(\gamma) h \bigr)
\]
is isomorphic to the total space of $\locsysH$. From this description, it follows
that the space of sections of $\locsysH$ over an open set $U \subseteq X$ is given by
\[
	\menge{\tilde{s} \colon \finv[p](U) \to H}{\text{$\tilde{s}(\gamma \cdot y) = 
			\rho(\gamma) \tilde{s}(y)$ for all $\gamma \in \pi_1(X, x)$, 
		$y \in \finv[p](U)$}}.
\]

\subsection*{Deligne's canonical extension}

Let $(\shV, \nabla)$ be a flat holomorphic vector bundle on a complex manifold $M$.
We assume that $M$ is an open subset of a bigger complex manifold $\Mb$, in such a
way that
\begin{enumerate}
\item $\Mb \setminus M$ has normal crossing singularities, and
\item $\nabla$ is unipotent along $\Mb \setminus M$.
\end{enumerate}
The second condition means that, near points of $\Mb \setminus M$, the local
monodromy for the local system of $\nabla$-flat sections should be unipotent.  Deligne proves (see \cite{Deligne}*{pp.~91--5} for
the precise statement) that
$(\shV, \nabla)$ admits a unique extension to a vector bundle $\shVb$ on $\Mb$, whose
defining property is that, in any local frame for $\shVb$, the connection $\nabla$
has only logarithmic poles along $\Mb \setminus M$ with nilpotent residues.

For the purposes of this paper, we need a description of $\shVb$ in local
coordinates, on a polydisk $\Delta^n$. Thus let $t_1, \dotsc, t_n$ be local
holomorphic coordinates near a point of $\Mb$, and assume that $\Mb \setminus M$ is
defined by the equation $t_1 \dotsm t_r = 0$. Restricting further, if necessary, it
suffices to treat the case when $M = \dstn{n}$.

Let $d$ be the rank of the bundle $\shV$, and $V$ its fiber at some basepoint in
$\dstn{n}$.  The fundamental group $\ZZ^n$ of $\dstn{n}$ acts on $V$, by parallel
translation, and we let $T_j$ be the operator corresponding to the $j$-th standard
generator of $\ZZ^n$. By assumption, each $T_j$ is a \emph{unipotent} operator, and
we can therefore define the nilpotent operators 
\[
	N_j = -\log T_j = \sum_{n = 1}^{\infty} \frac{1}{n} (\id - T_j)^n
\]
as their logarithms.\footnote{The minus sign is there to stay with the conventions
of other authors, for example \cite{CDK}.}

The vector bundle $\shVb$ has distinguished trivializations of the form
\begin{equation} \label{eq:trivialization}
	\shVb \simeq \shO_{\Delta^n} s_1 \oplus \dotsb \oplus \shO_{\Delta^n} s_d,
\end{equation}
for certain special sections $s_1, \dotsc, s_d$ of $\shV$ over $\dstn{n}$.
To obtain the sections in question, pull $(\shV, \nabla)$ back to the universal
covering space
\[
	p \colon \HH^n \to \dstn{n}, \qquad p(z_1, \dotsc, z_n) = 
		\bigl( e^{2\pi i z_1}, \dotsc, e^{2\pi i z_n} \bigr),
\]
where it becomes trivial (by virtue of being flat).
By our conventions, the fundamental group $\ZZ^n$ acts on $\HH^n$ by the rule
\[
	(a_1, \dotsc, a_n) \cdot (z_1, \dotsc, z_n) = (z_1 - a_1, \dotsc, z_n - a_n),
\]
and so sections of $\shV$ over $\dstn{n}$ correspond to holomorphic maps
$\tilde{s} \colon \HH^n \to V$ with the property that
\[
	\tilde{s}(z - e_j) = T_j \tilde{s}(z)
\]
for all $z \in \HH^n$ and all $j = 1, \dotsc, n$.

Now let $v_1, \dotsc, v_d \in V$ be an arbitrary basis for $V$. The maps
\[
	\tilde{s}_i \colon \HH^n \to V, \qquad \tilde{s}_i(z) = e^{\sum z_j N_j} v_i,
\]
have the required invariance property, because
\[
	\tilde{s}_i(z - e_j) = e^{-N_j} \tilde{s}_i(z) = T_j \tilde{s}_i(z),
\]
and thus define a frame of sections $s_1, \dotsc, s_d$ for $\shV$ on $\dstn{n}$.
These sections give the special trivialization of $\shVb$ in \eqref{eq:trivialization}.


\end{document}